\documentclass[12pt]{amsart}
\usepackage{geometry} 
\usepackage{pgfplots}
\usepackage{psfrag}
\usepackage{amsmath}
\usepackage{calc}
\usepackage{systeme}
\usepackage{amsfonts}
\usepackage{amsthm}
\usepackage{algorithm2e}
\usepackage[applemac]{inputenc}
\usepackage{bbold}
\usepackage[numbers]{natbib}
\newtheorem{hyp}{Hypothesis}

\newcommand{\R}{\mathbb{R}}
\newcommand{\be}{\begin{equation}}
\newcommand{\ee}{\end{equation}}
\newcommand{\ba}{\begin{aligned}}
\newcommand{\ea}{\end{aligned}}

\newtheorem{Theorem}{Theorem}
\newtheorem{Lemma}{Lemma}

\newtheorem{Cor}{Corollary}

\newcommand{\triplenorm}[1]{{\left\vert\kern-0.25ex\left\vert\kern-0.25ex\left\vert #1 
    \right\vert\kern-0.25ex\right\vert\kern-0.25ex\right\vert}}

\title{A two spaces extension of Cauchy-Lipschitz Theorem}

\begin{document}





\author{Charles Bertucci \textsuperscript{a},  Pierre-Louis Lions\textsuperscript{b,c}
}
\address{ \textsuperscript{a} CMAP, Ecole Polytechnique, UMR 7641, 91120 Palaiseau, France,\\ \textsuperscript{b} Universit\'e Paris-Dauphine, PSL Research University,UMR 7534, CEREMADE, 75016 Paris, France,\\ \textsuperscript{c}Coll\`ege de France, 3 rue d'Ulm, 75005, Paris, France. }

\maketitle

\begin{abstract}We adapt the classical theory of local well-posedness of evolution problems to cases in which the nonlinearity can be accurately quantified by two different norms. For ordinary differential equations, we consider $\dot{x} = f(x,x)$ for a function $f: V\times E \to E$ where $E$ is a Banach space and $V \hookrightarrow E$ a normed vector space. This structure allows us to distinguish between the two dependencies of $f$ in $x$ and allows to generalize classical results. We also prove a similar results for partial differential equations.
\end{abstract}



\setcounter{tocdepth}{1}

\tableofcontents
\section{Introduction}
This paper provides a simple extension of the classical theory of well-posedness of ordinary differential equations (ODE) in Banach spaces and Cauchy problems for partial differential equations (PDE). Among other things, our extension is valid when the dependence of the flow in the main argument can be decomposed into a dependence into two arguments taking values in different spaces. Particular instances of such an extension were used in Bertucci, Lasry and Lions \citep{bertucci2023lipschitz} to study mean field games master equations. We provide here a general version of these instances as we believe they can be of an independent interest.\\

The problems we are interested are either ODE of the form
\be\label{ode}
\dot x = f(x,x) \text{ on } (0,T), \quad x(0) = x_0 \in V,
\ee
or PDE of the form
\be\label{pde}
\partial_t u = F[u,u] \text{ on } (0,T) \times \Omega, \quad u|_{t = 0} = u_0 \in \Omega.
\ee
In the two previous equations, $T > 0$ is a fixed horizon time, $f : E\times V \to E$ where $V \subset E$ are two sets on which we are going to make assumptions later on and $F$ is an integro-differential operator defined on smooth functions $\Omega \to \R$ where $\Omega$ is some convex metric set. The main structure of problems we want to address here is a one in which the non-linearity can be accurately measured by means of two different norms of $x$ in the ODE setting, or two regularity norms in the PDE setting. The typical form of problems at interest is the one of non-linear transport equations of the form
\be\label{exode}
\partial_t u(t,x) = G(x,u(t,x))\cdot \nabla_x u(t,x) + g(x,u(t,x)) \text{ on } (0,T)\times \R^d,
\ee
with an initial condition $u_0$ which is for instance Lipschitz continuous on $\R^d$ and where $G$ and $g$ are two given functions.

In this setting we could consider $f(v,u) = G(x,v)\cdot \nabla_x u + g(x,u)$, with $V$ the set of Lipschitz functions on $\Omega$ and $E$ to be $L^\infty(\R^d,\R)$. This would place this problem in the scope of the ODE framework. However, generalizations of \eqref{exode} do not fall in the formalism of ODE, such as equations involving second order terms as
\[
\partial_t u(t,x) = \sum_{i,j = 1}^d a_{ij}(x) \partial_{ij}u(t,x) + G(x,u(t,x))\cdot \nabla_x u(t,x) + g(x,u(t,x)) \text{ on } (0,T)\times \mathbb{R}^d,
\]
where $(a_{ij}) = \sigma \sigma^\perp$ for some matrix valued function $\sigma$. We shall argue that, nonetheless, the two settings can be dealt with in exactly the same way. That is why we focus almost exclusively on the ODE setting and explain how to extend to the PDE framework afterwards.

\section{Assumptions and notation}

Let $(V, \triplenorm{\cdot})$ be a normed vector space and $(E,\|\cdot\|)$ a Banach space such that $V \subset E$. Let $f: V \times E \to E$. We assume the following. 
\begin{hyp}\label{hypEV}
If $(x_n)_{n \in \mathbb{N}}$ in $V$ is such that there exists $C \geq 0$ and $x \in E$,\\
$\lim_{n \to \infty} \|x_n - x \| = 0$ and $\sup_{n \in \mathbb{N}}\triplenorm{ x_n} \leq C$, then $x \in V$ and $\triplenorm x \leq \liminf_{n \to \infty} \triplenorm{x_n}$.
\end{hyp}

For any $1 \leq p \leq \infty$, $t > 0$, $x \in L^p([0,t],V)$, we note by $\triplenorm{x}_{p,t}$ its canonical norm, with the equivalent $\|x\|_{p,t}$ when $V$ is replaced by $E$.\\

We say that $x(\cdot) : [0,T)\to E$ is a solution of 
\[
\dot x = f(z,y) \text{ on } (0,T), \quad x(0) = x_0 \in E,
\]
if $s \to f(z(s),y(s)) \in L^1_{loc}([0,T),E)$ and for all $t \leq T$
\[
x(t) = \int_0^tf(z(s),y(s))ds + x_0.
\]
We shall make the two following assumptions on $f$.
\begin{hyp}\label{hyp1}
For all $T > 0, x_0 \in V$, $y \in L^\infty([0,T],V)$, there exists a unique solution $x \in L^\infty([0,T],V)\cap C([0,T],E)$ of 
\be\label{eqxy}
\dot x = f(y,x) \text{ on } (0,T), \quad x(0) = x_0.
\ee
Furthermore, for all $t \leq T$, this solution satisfies $\triplenorm{x(t)} \leq A(t,\triplenorm{x_0},\triplenorm{y(t)}_{\infty,t})$, where $A$ is a continuous function, non-decreasing in its three arguments, which does not depend on $y$ nor $x_0$, such that for all $M \geq 0$, $A(0,\triplenorm{x_0},M) = \triplenorm{x_0}$. 
\end{hyp}
\begin{hyp}\label{hyp2}
For any $x_0 \in V$, there exists $B,C : \mathbb{R}_+^2 \to \mathbb{R}_+$, depending only on $\triplenorm{x_0}$, satisfying, locally uniformly in $R >0$,
\be\label{hypbc}
\lim_{t\to 0} \frac{B(t,R)}{R}< 1,\quad \lim_{t\to 0} \frac{C(t,R)}{R} = 0,
\ee
such that for any $K > \triplenorm{x_0}$, there exists $t_0 > 0$ such that for all $(x_1,y_1)$ and $(x_2,y_2)$ satisfying
\[
\ba
\dot x_i = f(y_i,x_i) \text{ on } (0,t_0), \quad x_i(0) = x_0,\\
\forall t \leq t_0,\quad \|x_i(t)\|, \| y_i(t)\| \leq K,
\ea
\]
the following holds for $t \leq t_0$
\be\label{eqhyp2}
\|x_1-x_2\|_{\infty,t} \leq B(t,\|x_1-x_2\|_{\infty,t}) + C(t,\|y_1-y_2\|_{\infty,t}).
\ee
\end{hyp}

\section{Preliminaries}
Consider $x_0 \in V, R_0:=\triplenorm{x_0}$. We prove in this Section the following two Lemmas.
\begin{Lemma}
Under Hypothesis \ref{hyp1}, for all $K > R_0$, there exists $t_1 > 0$ depending only on $R_0$ and $K$ such that for $\triplenorm{y}_{\infty,t_1} \leq K$ implies that the unique solution of $x$ of \eqref{eqxy} satisfies $\triplenorm{x}_{\infty,t_1} \leq K$.
\end{Lemma}
\begin{proof}
It suffices to take $t_1 = \max\{t  > 0, A(t,R_0,K) \leq K\}$.
\end{proof}
\begin{Lemma}
Under Hypothesis \ref{hyp2}, for any $K > 0$, there exists $t_2 > 0$ and $\theta \in (0,1)$ depending only on $K$ and $R_0$ such that for any $(x_1,y_1),(x_2,y_2)$ as in Hypothesis \ref{hyp2}, for any $t \leq t_2$
\[
\|x_2-x_1\|_{\infty,t} \leq \theta \|y_2-y_1\|_{\infty,t}.
\]
\end{Lemma}
\begin{proof}
For $t_0$ given by Hypothesis \ref{hyp2},  $\|x_1 - x_2\|_{\infty,t_0} \leq 2K$. Hence, thanks to the assumption on $B$, choosing $t_2$ small enough, there exists $\theta_1 \in (0,1)$ such that for all $t \leq t_2$
\[
\|x_1 -x_2\|_{\infty,t} \leq \theta_1\|x_1-x_2\|_{\infty,t} +  C(t,\|y_1-y_2\|_{\infty,t}).
\]
Thus, there exists $\theta \in (0,1)$ such that taking $t_2$ so that for any $R \leq K$, $(1-\theta_1)^{-1}C(t_2,R) \leq R\theta$, we finally obtain that for all $t \leq t_2$
\[
\|x_2-x_1\|_{\infty,t} \leq \theta \|y_2-y_1\|_{\infty,t}.
\]
\end{proof}
The two Lemmas are concerned with \eqref{eqxy}. The first one yields some sort of boundedness for $\triplenorm{\cdot}$ while the second one yields some stability for $\|\cdot\|$.

\section{Main result}
Our main result is the following.
\begin{Theorem}\label{thm}
Under Hypotheses \ref{hypEV}, \ref{hyp1} and \ref{hyp2}, for any $x_0 \in V$, there exists $T_c \in (0,\infty]$ such that for all $T < T_c$, there exists a unique solution $x \in L^\infty([0,T],V)\cap C([0,T],E)$ of \eqref{ode} on $[0,T_c)$. Moreover, if $T_c < \infty$, then $\underset{t \to T_c}{\text{limess }} \triplenorm{x(t)} = \infty$.
\end{Theorem}
\begin{proof}
Take $K > \triplenorm{x_0}$, $t_1>0$ given by Lemma 1 and $t_2>0$ given by Lemma 2. If $t_2 > t_1$, set $t_2 = t_1$. Define, for any $t \in [0,t_2]$ $x_0(t) = x_0$ and for any $n \geq 0$, $x_{n+1}$ as the unique solution of 
\[
\dot{x}_{n+1} = f(x_n,x_{n+1}) \text{ on } [0,t_2], \quad x_{n+1}(0) = x_0.
\]
Thanks to Lemma 1, for all $n \geq 0$, $\triplenorm{x_n}_{\infty,t_2} \leq K$. Hence, we deduce from Lemma 2 that there exists $\theta \in (0,1)$ such that for all $n \geq 0$
\[
\|x_{n+2}-x_{n+1}\|_{\infty,t_2} \leq \theta \|x_{n+1}-x_n\|_{\infty,t_2}.
\]
Hence, $(x_n)_{n \geq 0}$ is a Cauchy sequence in $L^\infty([0,t_2],E)$, thus it converges toward some limit $x \in L^\infty([0,t_2],E)$. For almost every $t \in [0,t_2]$, $(x_n(t))_{n \geq 0}$ is a bounded sequence of $V$ converging toward $x(t)$. Hence, from Hypothesis \ref{hypEV}, $x(t) \in V$ and $\triplenorm{x(t)} \leq K$, for almost every $t \in [0,t_2]$.

From Hypothesis 1, we can consider $z$, the unique solution of 
\[
\dot{z} = f(x,z) \text{ on } [0,t_2], \quad z(0) = x_0.
\]
From Lemma 2, $\|z-x_{n+1}\|_{\infty,t_2} \leq \theta \|x- x_n\|_{\infty,t_2}$. Passing to the limit $n \to \infty$, we obtain that $z = x$. Hence $x$ is a solution of \eqref{ode} on $[0,t_2]$. The uniqueness of such a solution on $[0,t_2]$ trivially follows from Lemma 2.\\

Note that $t_2$, which was given by Lemma 2, only depends on $K$, hence on $\triplenorm{x_0}$. From this we deduce that the existence time we obtained is bounded from below on balls of $V$. This implies that we can repeat the same argument by replacing $x_0$ with $x(t_2)$. We then arrive at the construction of a solution on $[t_2,t_3]$ for some $t_3 > t_2$ depending only on $\triplenorm{x(t_2)}$. By gluing the two solutions, we then have constructed a solution on $[0,t_3]$, which is clearly unique in view of Lemma 2. Repeating the argument, we build a sequence of increasing times $(t_n)_{n \geq 2}$ such that $x$ is the unique solution of \eqref{ode} on $[0,t_n]$ for all $n \geq 2$. Note $T_c :=\lim_{n \to \infty} t_n$. If $T_c = \infty$, the result is proved. If not, assume that $M :=\triplenorm{x}_{\infty,T_c} < \infty$. Then, from the first part of the proof, there exists $\delta >0$ depending only on $M$, such that for any $t \in [0,T_c)$, we can extend the solution $x$ to $[t,t+\delta]$. Up to changing $T_c$ into $T_c + \delta$, this yields an alternative: either $T_c = \infty$, or $\underset{t \to T_c}{\text{limess }} \triplenorm{x(t)} = \infty$.
\end{proof}
An immediate Corollary of Lemma 2 leads to the following version of our result.
\begin{Cor}
The conclusions of Theorem \ref{thm} hold true if in Hypothesis \ref{hyp2}, we exchange $B$ and $C$ in \eqref{hypbc}.
\end{Cor}

\section{The case of partial differential equations}
In the case of problems of the form \eqref{pde}, we can have exactly the same approach. It consists in assuming that Hypothesis \ref{hypEV} holds and that for any $u_0 \in V$, we can define for any $T > 0$ the solution operator $\Psi$ which associates to $v \in L^\infty([0,T],V)$, the unique solution $\Psi(v) \in L^\infty([0,T],V)\cap C([0,T],E)$ of 
\[
\partial_t \Psi(v) = F[v,\Psi(v)]  \text{ in } (0,T)\times \Omega,
\]
which has to be understood in a sense which is context dependent. We shall then say that $u$ is a solution of \eqref{pde} on $(0,T)$ if $\Psi(u) = u$.

The analogue of Hypotheses \ref{hyp1} and \ref{hyp2} is this context is the following.
\begin{hyp}\label{hypPDE}
Take $T >0$ and $u_0 \in V$. The operator $\Psi$ is well defined and there exists $A$, $B$ and $C$ as in Hypotheses \ref{hyp1} and \ref{hyp2} such that 
\begin{itemize}
\item For any $v \in L^\infty([0,T],V)$, for all $t \leq T, \triplenorm{\psi(t)} \leq A(t,\triplenorm{u_0},\triplenorm{v}_{\infty,t})$.
\item For any $K > \triplenorm{u_0}$, there exists $t_0 > 0$ such that for $v_1,v_2 \in L^\infty([0,T],V)$ such that $\|v_1\|_{\infty,t_0},\|v_2\|_{\infty,t_0},\|\Psi(v_1)\|_{\infty,t_0},\|\Psi(v_1)\|_{\infty,t_0} \leq K$, 
\[
\|\Psi(v_1)-\Psi(v_2)\|_{\infty,t} \leq B(t,\|\Psi(v_1)-\Psi(v_2)\|_{\infty,t}) + C(t,\|v_1-v_2\|_{\infty,t}).
\]
\end{itemize}
\end{hyp}
Following step by step the arguments of the previous case, we arrive at the
\begin{Theorem}
Under Hypotheses \ref{hypEV} and \ref{hypPDE}, for any $u_0 \in V$, there exists $T_c \in (0,\infty]$ such that for all $T < T_c$, there exists a unique solution $u \in L^\infty([0,T],V)\cap C([0,T],E)$ of \eqref{pde} on $[0,T_c)$. Moreover, if $T_c < \infty$, then $\underset{t \to T_c}{\text{limess }} \triplenorm{u(t)} = \infty$.
\end{Theorem}

\section{Comments}

\subsection{On the assumptions}
We briefly comment on the Hypothesis \ref{hyp1} and \ref{hyp2}. We believe our framework to be particularly helpful when the equation \eqref{eqxy} is an \textit{easy} equation to solve in $x$.\\

Hypothesis \ref{hyp1} is typically obtained by proving an a priori estimate on the solution of \eqref{eqxy}. Formally, it states that the regularity of the solution $x$ of \eqref{eqxy} deteriorates continuously with the time, so that, for sufficiently small times, it can stay bounded.\\

Hypothesis \ref{hyp2} is typically a stability estimate on the solution of \eqref{eqxy} with respect to variations of $y$. It is stated here in a quite general way. Note that Hypothesis \ref{hyp2} is verified for instance if \eqref{eqhyp2} is replaced by 
\[
\|x_1(t)-x_2(t)\| \leq K_0(\|x_1-x_2\|_{p,t} + \|y_1-y_2\|_{q,t}),
\]
for $1 \leq p,q < \infty$ and $K_0 > 0$, namely thanks to the inequality $\|z\|_{p,t} \leq t^{\frac 1p} \|z\|_{\infty,t}$.

\subsection{Natural extension}
Note that Theorem \ref{thm} trivially extends to the case of 
\[
\dot x = f(t,x,x) \text{ on } (0,T), \quad x(0) = x_0 \in V,
\]
given that, for any $\delta \geq 0$, $f(\cdot+\delta,\cdot,\cdot)$ satisfies Hypotheses \ref{hyp1} and \ref{hyp2}.

\section*{Acknowledgments}
The authors have been partially supported by the Chair FDD/FIME (Institut Louis Bachelier). The first author has been partially supported by the Lagrange Mathematics and Computing Research Center.
\bibliographystyle{plainnat}
\bibliography{bibmatrix}

\end{document}